\documentclass[11pt]{article}

\usepackage{amsmath,amssymb,amsthm,amscd,epsfig}
\newcommand{\vp}{\varepsilon}

\theoremstyle{plain}

\newtheorem{thm}{Theorem}

\newtheorem{lem}{Lemma}
\newtheorem{cor}{Corollary}

\newtheorem{pro}{Proposition}

\theoremstyle{definition}

\newtheorem{defn}{Definition}

\theoremstyle{remark}

\begin{document}

\title{Invariant measures involving local inverse iterates}

\author{Eugen Mihailescu}
\date{}
\maketitle

\begin{abstract}

We study new invariant probability measures, describing the
distribution of multivalued inverse iterates (i.e of different
local inverse iterates) for a non-invertible smooth function $f$
which is hyperbolic, but not necessarily expanding
 on a repellor $\Lambda$. The methods for the higher dimensional non-expanding and non-invertible case are different than the ones for
 diffeomorphisms, due to the lack of a nice unstable foliation (local unstable manifolds depend on prehistories and may
intersect each other, both in $\Lambda$ and outside $\Lambda$),
and the fact that Markov partitions may not exist on $\Lambda$. We
obtain that for Lebesgue almost all points $z$ in a neighbourhood
$V$ of $\Lambda$, the normalized averages of Dirac measures on the
consecutive preimage sets of $z$ converge weakly to an equilibrium
measure $\mu^-$ on $\Lambda$; this implies that $\mu^-$ is a
physical measure for the local inverse iterates of $f$. It turns
out that $\mu^-$ is an inverse SRB measure in the sense that it is
the only invariant measure satisfying a Pesin type formula for the
negative Lyapunov exponents. Also we show that $\mu^-$ has
absolutely continuous conditional measures on local stable
manifolds, by using the above convergence of measures. Several
classes of examples of hyperbolic non-invertible and non-expanding
repellors, with their inverse SRB measures, are given in the end.

\end{abstract}



\section{Introduction}\label{sec1}

SRB measures (Sinai, Ruelle, Bowen) and physical measures have
been studied for many classes of dynamical systems having some
form of hyperbolicity, either uniform, partial or non-uniform.
Intuitively physical measures describe the distributions of
forward iterates in a neighbourhood of an attractor.
 For uniformly hyperbolic systems in the vicinity of an attractor, the existence of physical measures and many
 of their properties were proved by Bowen (\cite{Bo}). SRB measures are usually defined by the absolute continuity
 of their conditional measures on local unstable manifolds (\cite{Y}). The term \textbf{physical measures} was introduced
 by Eckmann and Ruelle (\cite{ER}) who also proved many of their properties and gave relations to examples from physics
  (turbulence theory, statistical mechanics, strange attractors, etc.) Measure-theoretic entropy and Lyapunov
  exponents prove to be very important with regard to physical and SRB measures, as in Pesin's entropy formula
  (\cite{ER}, \cite{LY}, \cite{PDL}, \cite{Y}, etc.) For uniformly hyperbolic dynamical systems having an
  attractor $\Lambda$, physical measures are in fact SRB measures as was proved by Sinai, Ruelle, Bowen
  (\cite{Bo}, \cite{Y}). For other systems there may exist physical measures which are not SRB (as in \cite{ER}).
In \cite{ER}, it was studied mainly the case of attractors for diffeomorphisms or the case of a flow indexed with
 both positive and negative parameters $t$. In such a case the inverse of the map is well defined and it is also a
 smooth map. For flows we simply can take $f^t, t <0$. One cannot do the same if the dynamical system is not invertible.

In this paper we focus on finding physical measures giving the
\textbf{distribution of consecutive preimage sets} for
non-invertible smooth maps (such maps will be called
\textbf{endomorphisms}), in the vecinity of a hyperbolic repellor.
There are many examples of systems which are not invertible, for
instance the non-invertible horseshoes from \cite{Bot},
s-hyperbolic holomorphic maps in several dimensions and their invariant sets (\cite{Mi}), skew products
having a finite iterated function system in the base and overlaps
in their fibers, hyperbolic toral endomorphisms, baker's
transformations with overlaps, etc. By similarity to the SRB measure (\cite{Bo}, \cite{Y}), one natural question would be to study the
\textbf{distribution of various preimages near a hyperbolic repellor} $\Lambda$.
The problem is that there is \textbf{no} unique inverse $f^{-1}$;
instead, if $f$ does not have any critical points near the
$\Lambda$, we will obtain local inverse iterates, or equivalently
a \textbf{multivalued inverse iterate} of $f$. If $f$ is locally
$d$-to-1 on a basic set $\Lambda$, and if the local inverse
iterates of $f$ on some open set $W$ are denoted by $f^{-1}_{W,
1}, \ldots, f^{-1}_{W, d}$, then the multivalued inverse of $f$ on
$W$ is $(f^{-1}_{W, 1}, \ldots, f^{-1}_{W, d})$. Knowing the
behaviour of inverse trajectories of a system may be
important when we want to obtain information about the past
states of the system.

It is important to keep in mind that the map $f$ \textbf{is not assumed
expanding} on $\Lambda$; indeed for the expanding case a lot is
known about the distribution of preimages (see \cite{Ma}, \cite{Ru-exp}) and the situation is characterized by the fact
that local inverse iterates decrease exponentially fast the diameter of small balls;
this guarantees that we have bounded distortion lemmas. However in the general
higher dimensional non-invertible hyperbolic case we do not have control on the distortion of small balls under local inverse
iterates; indeed they may increase in the stable direction in backward time.

Non-invertibility brings many difficulties into the setting, like not being able to apply directly Birkhoff
Ergodic Theorem for $f^{-1}$ like in the case of diffeomorphisms, the non-existence of a Markov partition of
 $\Lambda$ (as $f$ is just an endomorphism, not necessarily expanding on $\Lambda$), etc. One classical tool when dealing with endomorphisms would be
 to use the natural extension $\hat \Lambda$ of $\Lambda$ (also known as the inverse limit), but then one looses differentiability properties near
 $\Lambda$, as $\hat \Lambda$ is not a manifold.
  In general for endomorphisms, local unstable manifolds depend on whole prehistories not only on the base points
  (\cite{R}); this dependence is Holder continuous with respect to prehistories (\cite{M}).
   Our repellors will be in fact unions of global stable sets, but the overlappings and foldings of the system introduce a complicated and very irregular dynamics.
   Moreover the number of preimages belonging to $\Lambda$ of a given point may vary a priori along $\Lambda$.

For attractors/repellors $\Lambda$ for diffeomorphisms $f$ we know that there exists an SRB/inverse SRB measure on $\Lambda$ and that $(\Lambda, f|_\Lambda)$ becomes a Bernoulli 2-sided transformation (\cite{Bo}, \cite{Ma}). This is based mainly on the existence of Markov partitions in the invertible case. Also for expanding maps there exist Markov partitions (\cite{Ru-carte}, \cite{Ru-exp}) and the system is isomorphic to a 1-sided Markov chain. In the non-invertible non-expanding case we however do not have Markov partitions, as mentioned above.

The \textbf{main directions and results} of the paper are the following:

First we will specify what we understand by a repellor, in
Definition \ref{repellor}. We prove that on a repellor $\Lambda$, the number of preimages belonging to $\Lambda$ of any $x \in \Lambda$ is locally constant. We also show a very important property
of these sets, namely the stability under perturbations, in
Proposition \ref{stability}. Then we prove in Theorem
\ref{pressure} that the pressure of the stable potential $\Phi^s$
 along a connected repellor $\Lambda$ is related to the number $d$ of
preimages of an arbitrary point, which remain in $\Lambda$.

We will define next the probability measures $$\mu^z_n:= \frac
{1}{d^n} \mathop{\sum}\limits_{y \in f^{-n}z \cap U} \frac 1n
\mathop{\sum}\limits_{i=0}^{n-1} \delta_{f^i y}, n \ge 1, z \in V
\subset U,$$ where $V, U$ are close enough neighbourhoods of
$\Lambda$. In Theorem \ref{physical} we give the main result,
namely the weak convergence of the measures $\mu_n^z$ towards the
equilibrium measure $\mu_s$ of the potential $\Phi^s$, for
Lebesgue almost all points $z \in V$. In this Theorem, $\Lambda$
will be assumed connected (not a very restrictive assumption for
our notion of repellor, as will be seen). We show then in Theorem
\ref{cond} that a Pesin type formula involving the
\textit{negative Lyapunov exponents} can be derived for this

physical measure $\mu^-= \mu_s$. This will give also the absolute
continuity of conditional measures of $\mu^-$ on stable manifolds,
by using the convergence of measures of Theorem \ref{physical} and
a result of Liu (\cite{PDL}) relating entropy, folding entropy and
negative Lyapunov exponents. In fact by using the convergence of
the measures $(\mu^z_n)_n$ from Theorem \ref{physical}, we show
that the folding entropy $H_{\mu^-}(\epsilon/f^{-1}\epsilon)$ is
equal to $\log d$, where $\epsilon$ is the partition of $\Lambda$
into single points. Therefore by all these properties, it follows
that $\mu^-$ can be viewed as an \textbf{inverse SRB measure}.

The above inverse Pesin type formula will imply in Theorem
\ref{Bernoulli} that the repellor $\Lambda$ with its inverse SRB
measure $\mu^-$ is not isomorphic to a one-sided Bernoulli shift.
This is in contrast with the case of attractors for
diffeomorphisms where the attractor, together with its SRB
measure, is 2-sided Bernoulli. We show however in Theorem
\ref{exact} that $\mu^-$ has Exponential Decay of Correlations on
Holder potentials.

Finally we describe some classes of examples in Section
\ref{sec3}, among which hyperbolic toral endomorphisms, other
Anosov endomorphisms, as well as new classes of non-expanding
repellors which are not Anosov, together with their inverse SRB
measures.

\section{Main results.}\label{sec2}

First we will specify what do we understand by \textit{repellor}.
As a \textbf{general setting} throughout the paper, we consider $f:M \to M$ a smooth (say $\mathcal{C}^2$) map on a Riemannian manifold,
and $\Lambda$ an $f$-invariant compact set in $M$ which does not intersect the critical set $C_f$ of $f$. We remark that the preimages of a point from $\Lambda$ do not have to remain in
$\Lambda$ necessarily. Also let us notice that if $C_f$ would intersect $\Lambda$, the basic ideas would remain the same
as long as  we assume an integrability condition on $\log |Df_s|$ over $\Lambda$.

\begin{defn}\label{repellor}

Let $f: M \to M$ be a smooth (for example $\mathcal{C}^2$) map on a Riemannian manifold and let $\Lambda$ be a compact set
which is $f$-invariant (i.e $f(\Lambda) = \Lambda$) and s.t $f|_\Lambda$ is topologically transitive; \ assume also that there exists a neighbourhood $U$ of $\Lambda$ such that $\Lambda = \mathop{\cap}\limits_{n \in \mathbb Z} f^n U$. Such a set will be called a \textit{basic set} for $f$ (\cite{KH}). We say that
 $\Lambda$ is a \textbf{repellor} for $f$ if $\Lambda$ is a basic set for $f$, $C_f \cap \Lambda = \emptyset$ and if there exists a neighbourhood $U$ of $\Lambda$ such that $\bar U \subset f(U)$. $\hfill\square$

\end{defn}

We will call any point $y \in f^{-1}(x)$ an $f$-\textit{preimage} of $x\in M$; and by $n$-\textit{preimage}
of $x$ we mean any point $y \in f^{-n}(x)$, for an integer $n >0$.

\begin{pro}\label{constant}

In the setting of Definition \ref{repellor}, if $\Lambda$ is a repellor for $f$, then $f^{-1}\Lambda \cap U = \Lambda$.
If moreover $\Lambda$ is assumed to be connected, the number of $f$-preimages that a point has in $\Lambda$ is constant.
\end{pro}

\begin{proof}

Let a point $x \in \Lambda$, and $y$ be an $f$-preimage of $x$ from $U$.
Then $f^n y \in \Lambda, n \ge 1$. From Definition \ref{repellor}, since $\Lambda$ is assumed to be a repellor,
the point $y$ has a preimage $y_{-1}$ in $U$; then $y_{-1}$ has a preimage $y_{-2}$ from $U$, and so on.
Thus $y$ has a full prehistory belonging to $U$ and also its forward orbit belongs to $U$, hence $y \in \Lambda$ since $\Lambda$ is a basic set.
So $f^{-1} \Lambda \cap U = \Lambda$.

We prove now the second part of the statement. Let a point $x \in \Lambda$ and assume that it has $d$ $f$-preimages in $\Lambda$, denoted $x_1, \ldots, x_d$. Consider also another point $y \in \Lambda$ close to $x$. If $y$ is close enough to $x$ and since $\mathcal{C}_f \cap \Lambda = \emptyset$, it follows that $y$  also has exactly $d$ $f$-preimages in $U$, denoted by $y_1, \ldots, y_d$.  Since from the first part we know that $f^{-1}\Lambda \cap U = \Lambda$, we obtain that $y_1, \ldots, y_d \in \Lambda$.
In conclusion the number of $f$-preimages in $\Lambda$ of a point is locally constant. If $\Lambda$ is assumed to be connected, then the number of preimages belonging to $\Lambda$ of any point from $\Lambda$, must be constant.
\end{proof}

Let us denote by $d(x)$ the number of $f$-preimages that the point $x$ has in the repellor $\Lambda$. Then from the above Proposition we know that $d(\cdot)$ is locally constant on $\Lambda$. Clearly there exist only finitely many values that $d(\cdot)$ may take on $\Lambda$. We will assume in the sequel that the number of preimages $d(\cdot)$ is \textbf{constant} on $\Lambda$.
This happens for instance when $\Lambda$ is connected (from Proposition \ref{constant}). We give the results in this setting (i.e when $\Lambda$ is connected), but in fact all we need is that $d(\cdot)$ is constant.

We will work with uniformly hyperbolic endomorphisms on $\Lambda$ (\cite{R}, \cite{Bot}, \cite{M}, etc.)
The stable tangent spaces $E^s_x, x \in \Lambda$ depend Holder continuously on $x$ (see \cite{KH}, \cite{M}, \cite{Mi1});
the unstable
 tangent spaces depend on whole prehistories, i.e we have $E^u_{\hat x}, \hat x \in \hat \Lambda$.
  Here $(\hat \Lambda, \hat f)$ is the \textbf{natural extension} (\cite{Ro}), or \textbf{inverse limit} of the dynamical system $(\Lambda, f)$; the space $\hat \Lambda:= \{\hat x = (x, x_{-1}, x_{-2}, \ldots), f(x_{-i}) = x_{-i+1}, i \ge 1, x _0:=x\}$ is the space of full prehistories of points from $\Lambda$ and the map $\hat f: \hat \Lambda \to \hat \Lambda, \hat f(\hat x) = (fx, x, x_{-1}, x_{-2}, \ldots), \hat x \in \hat \Lambda$ is the \textit{shift homeomorphism}.
We denote also by $\pi:\hat \Lambda \to \Lambda$ the \textit{canonical projection} given by $\pi(\hat x) = x, \hat x \in \hat \Lambda$. The compact topological space $\hat\Lambda$ can be endowed with a natural metric, but it is not a manifold.

We shall denote $Df|_{E^s_x}$ by $Df_s(x)$ and call it the \textit{stable
derivative} at $x\in \Lambda$; and $Df_u(\hat x):= Df|_{E^u_{\hat x}}$ is the
\textit{unstable derivative} at $\hat x \in \hat \Lambda$. Similarly the local stable and unstable manifolds are denoted by $W^s_r(x),
W^u_r(\hat x), \hat x \in \hat \Lambda$, for some small $r>0$. We call  \textbf{stable potential} the function
$$\Phi^s(x):= \log |Jac(Df_s(x))| = \log |\text{det}(Df_s(x))|, x \in \Lambda$$

One notices that there exists a bijection between the set $\mathcal{M}(f)$ of $f$-invariant probability
measures on $\Lambda$ and the set $\mathcal{M}(\hat f)$ of $\hat f$-invariant probability measures on the natural
extension $\hat \Lambda$, so that to any measure $\mu \in \mathcal{M}(f)$ we associate the unique measure
$\hat \mu \in \mathcal{M}(\hat f)$ satisfying the relation $\pi_*(\hat \mu) = \mu$ (for example Rokhlin, \cite{Ro}).
 It is easy to show that $h_{\hat \mu}(\hat f) = h_\mu(f)$ and that $P_{\hat f}(\phi\circ \pi) = P_f(\phi),
 \forall \phi \in \mathcal{C}(\Lambda, \mathbb R)$. Thus $\mu$ is an equilibrium measure for a potential $\phi$ if and only if its unique $\hat f$-invariant lifting $\hat \mu$ is an equilibrium measure for $\phi \circ \pi$ on $\hat \Lambda$.
Next let us transpose to the setting of endomorphisms, some properties of equilibrium measures from the diffeomorphism case, by using liftings to the natural extension.

\begin{pro}\label{eq-endo}

Let $\Lambda$ be a hyperbolic basic set for a smooth endomorphism $f:M \to M$, and let $\phi$ a Holder
continuous function on $\Lambda$. Then there exists a unique equilibrium measure
$\mu_\phi$ for $\phi$ on $\Lambda$ such that for any $\vp >0$, there exist positive constants $A_\vp, B_\vp$ so that for any $y \in \Lambda, n \ge 1$,
$$ A_\vp e^{S_n \phi(y) - nP(\phi)}  \le \mu_\phi(B_n(y,
\vp)) \le  B_\vp e^{S_n\phi(y)-n P(\phi)} $$
\end{pro}

\begin{proof}

The shift $\hat f :\hat \Lambda \to \hat \Lambda$ is an expansive homeomorphism. The existence of a unique
equilibrium measure for the Holder potential $\phi\circ \pi$ with respect to the homeomorphism
$\hat f:\hat \Lambda \to \hat \Lambda$ follows from the standard theory of expansive homeomorphisms
(for example \cite{KH}); let us denote it by $\hat \mu_\phi$.
According to the discussion above there exists a unique probability measure $\mu_\phi$ with $\mu_\phi:= \pi_* \hat \mu_\phi$, and $\mu_\phi$ is the unique equilibrium measure for $\phi$ on $\Lambda$. The uniqueness follows from the bijection between $\mathcal{M}(f)$ and $\mathcal{M}(\hat f)$ and from the
fact that $\hat\phi:= \phi\circ\pi:\hat \Lambda \to
\mathbb R$ is Holder continuous (as $\pi:\hat \Lambda \to \Lambda $ is Lipschitz and $\phi$ is Holder).
Now, there exists a $k=k(\vp) \ge 1$ such that
 $\hat f^k(\pi^{-1}B_n(y, \vp)) \subset B_{n-k}(\hat f^k\hat y, 2\vp) \subset \hat \Lambda$, for any $y \in \Lambda$.
 On the other hand for any $\hat y \in \hat \Lambda$, we have $\pi(B_n(\hat y, \vp)) \subset B_n(y, \vp)$.
 The last two set inclusions and the $\hat f$-invariance of $\hat \mu_\phi$, together with the estimates for the
 $\hat \mu_\phi$-measure of the Bowen balls in $\hat \Lambda$ (from \cite{KH}) imply that there exist positive
 constants $A_\vp, B_\vp$ (depending on $\vp>0$ and $\phi$)
 such that the estimates from the statement hold.
\end{proof}

Next let us show that the notion of connected repellor is \textbf{stable under perturbations}; this property is important when dealing with systems having a small level of random noise, as it happens in most physical situations.

\begin{pro}\label{stability}

Let $\Lambda$ be a connected repellor for a smooth map $f: M \to
M$ so that $f$ is hyperbolic on $\Lambda$, and let a perturbation
$g$ which is  $\mathcal{C}^1$-close to $f$. Then $g$ has a
connected repellor $\Lambda_g$ close to $\Lambda$ such that $g$ is
hyperbolic on $\Lambda_g$. In addition the number of $g$-preimages
belonging to $\Lambda_g$ of any point of $\Lambda_g$, is the same
as the number of $f$-preimages in $\Lambda$ of a point from
$\Lambda$.
\end{pro}

\begin{proof}

Since $\Lambda$ has a neighbourhood $U$ so that $\bar U \subset f(U)$, it follows that for $g$ close enough to $f$, we will obtain $\bar U \subset g(U)$.
If $g$ is $\mathcal{C}^1$-close to $f$, then we can take the set
$$
\Lambda_g:= \mathop{\cap}\limits_{n \in \mathbb Z} g^n(U)
$$
and it is quite standard that $g$ is hyperbolic on $\Lambda_g$ (for example \cite{R}, \cite{M}, etc.)
One can form then the natural extension of the system $(\Lambda_g, g)$.
We know that there exists a conjugating homeomorphism $H: \hat \Lambda \to \hat \Lambda_g$ which commutes with $\hat f$ and $\hat g$. The natural extension $\hat \Lambda$ is connected if $\Lambda$ is connected, from the fact that the topology on $\hat \Lambda$ is induced by the product topology from $\Lambda^{\mathbb N}$. Hence $\hat \Lambda_g$ is connected and thus $\Lambda_g$ itself is connected too.
Moreover we have that $\bar U \subset g(U)$ if $g$ is close enough to $f$, thus
$\Lambda_g$ is a connected repellor for $g$.

Now for the second part of the proof, assume that $x \in \Lambda$ has $d$ $f$-preimages in $\Lambda$. Then if $C_f \cap \Lambda = \emptyset$ and if $g$ is $\mathcal{C}^1$-close enough to $f$, it follows that the local inverse iterates of $g$ are close to the local inverse iterates of $f$ near $\Lambda$. Thus any point $y \in \Lambda_g$ has exactly $d$ $g$-preimages in $U$, denoted by $y_1, \ldots, y_d$. Any of these $g$-preimages from $U$ has also a $g$-preimage in $U$ since $\bar U \subset g(U)$, and so on. This implies that $y_i \in \Lambda_g = \mathop{\cap}\limits_{n \in \mathbb Z} g^n(U), i =1, \ldots, d$;  hence $y$ has exactly $d$ $g$-preimages belonging to $\Lambda_g$.
\end{proof}

We will need in the sequel an estimate of the volume of a \textit{tubular unstable neighbourhood} $f^n(B_n(y, \vp))$,
where $B_n(y, \vp):= \{z \in M, d(f^i z, f^i y) < \vp, i = 0, \ldots, n-1\}$ is a Bowen ball.  The set $f^n(B_n(y, \vp))$ is a neighbourhood in $M$ of the local unstable manifold $W^u_\vp(\hat f^n y)$, for $\hat f^n y = (f^n y, f^{n-1}y, \ldots, y, \ldots)$. Such sets were used in the definition of the inverse pressure, a notion developed in order to obtain estimates for the stable dimension in the non-invertible case (\cite{Mi1}).

By the measure $m(\cdot)$ on $M$ we understand the
\textbf{Lebesgue measure} defined on the manifold $M$. And by
$S_n\phi(y)$ we denote the consecutive sum $\phi(y) + \ldots +
\phi(f^{n-1}y)$ for  $y \in \Lambda, \phi \in \mathcal{C}(\Lambda,
\mathbb R)$.

\begin{lem}\label{tub}

Let $f:M \to M$ be a smooth endomorphism and $\Lambda$ be a basic set on which $f$ is hyperbolic. Then for some fixed small $\vp>0$ there exist positive constants $A, B >0$ such that for any $n \ge 1$ we have:
$$
A e^{S_n\Phi^s(y)} \le m(f^n B_n(y, \vp)) \le B e^{S_n\Phi^s(y)}
$$
\end{lem}

\begin{proof}

First of all let us notice that $S_n\Phi^s(y) = \log |\text{det}(Df_s^n(y))|, y \in \Lambda, n \ge 1$.
From \cite{KH} we know that the stable spaces depend Holder continuously on their base point.
Thus $\Phi^s$ is a Holder function on $\Lambda$, as $\mathcal{C}_f \cap \Lambda = \emptyset$. Thus as in proposition 1.6 from \cite{Mi1}, we obtain a Bounded
Distortion Lemma, saying that there exist positive constants $\tilde A, \tilde B$ such that
$
\tilde A \le \frac{e^{S_n\Phi^s(z)}}{e^{S_n\Phi^s(y)}} \le \tilde
B, n \ge 1, z \in B_n(y, \vp).
$
Then using this Bounded Distortion Lemma, the conclusion follows similarly as in \cite{Bo}.
\end{proof}

\begin{thm}\label{pressure}

Consider $\Lambda$ to be a connected hyperbolic repellor for the
smooth endomorphism $f : M \to M$; let us assume that the constant number of $f$-preimages belonging to $\Lambda$
of any point from $\Lambda$ is equal to $d$.
Then $P(\Phi^s - \log d) = 0$.
\end{thm}

\begin{proof}

As we have seen in Proposition \ref{constant} if $\Lambda$ is a connected repellor, then the number of preimages belonging to $\Lambda$ of any point from $\Lambda$ is constant and equal to some integer $d>0$.

In fact if the neighbourhood $V$ of $\Lambda$ is close enough to $\Lambda$, then we can assume that any point $y \in V$ has exactly $d$ $f$-preimages belonging to $U$. We want to show that there exists a neighbourhood $V$ of $\Lambda$ such that any point from $V$ has exactly $d^n$ $n$-preimages belonging to $U$, for any $n \ge 1$.
First let us assume that the metric around $\Lambda$ is adapted to the hyperbolic structure on $\Lambda$, i.e there is $\lambda \in (0, 1)$ so that if $z \in W^u_r(\hat y)$ and $\hat z = (z, z_{-1}, \ldots)$ is the prehistory of $z$ $r$-shadowing the prehistory $\hat y$, then

\begin{equation}\label{unst}
d(y, z) \ge d(y_{-1}, z_{-1})\cdot \frac {1}{\lambda} \ge d(y_{-2}, z_{-2}) \cdot \frac {1}{\lambda^2}\ge \ldots
\end{equation}

Now consider a point $y \in V$ and some preimage $y_{-1} \in f^{-1}(y) \cap U$; If $y$ is close enough to $\Lambda$, then
$y_{-1} \in U$, and let us assume that we can continue this prehistory until we reach level $m$.
In other words $(y, y_{-1}, \ldots, y_{-m})$ is a finite prehistory of $y$ with $y_{-1}, \ldots, y_{-m} \in U$,
but there exists a preimage $\tilde y_{-m-1}$ of $y_{-m}$ which escapes $U$, so that $y_{-m}$ has less than $d$ preimages
in $U$.
From the definition of repellor we know that $\bar U \subset f(U)$, thus there exists some preimage $y_{-m-1} \in U \cap f^{-1}(y_{-m})$. Then this preimage $y_{-m-1}$ will have a full prehistory in $U$. Since $\Lambda$ is a basic set and $y_{-m}$ has a full prehistory in $U$, it follows that there exists a
prehistory $\hat \xi \in \hat \Lambda$ such that $y_{-m} \in W^u_r(\hat \xi)$, if $U$ is close enough to $\Lambda$ (\cite{KH}, \cite{R}).

Consequently $y \in W^u_r(\hat f^m \hat \xi)$; but from (\ref{unst}) we have $d(y_{-m}, \Lambda) \le d(y_{-m}, \xi)
\le \lambda^m d(y, f^m\xi) \le \lambda^m d(y, \Lambda)$. Recall however that a preimage of $y_{-m}$ escapes $U$, thus $d(y_{-m}, \Lambda)$ must be larger than some positive fixed constant $\chi_0$. Therefore if $V$ is close enough to $\Lambda$ (and hence $m$ is
large enough) we obtain a contradiction, since we know from above that $d(y_{-m}, \Lambda) \le \lambda^m \cdot d(y, \Lambda)$.

Hence there must exist a neighbourhood $V$ of $\Lambda$ such that any point from $V$ has exactly $d^n$ $n$-preimages belonging to $U$, for any $n \ge 1$.

Let us take now an $(n, \vp)$-separated set of maximal cardinality in $\Lambda$ and denote it by $F_n(\vp)$.
Hence $B_n(y, \vp/2) \cap B_n(z, \vp/2) = \emptyset, \forall y, z \in F_n(\vp)$. From the maximality condition it
follows also that $\Lambda \subset \mathop{\cap}\limits_{y \in F_n(\vp)} B_n(y, 2\vp)$.
Now from the fact that $C_f \cap \Lambda = \emptyset$, it follows that there exists a positive constant $\vp_0$ such that if $y, z \in f^{-1}x \cap U, y \ne z$, then $d(y, z) > \vp_0$.
This implies that if $y, z \in f^{-n}x \cap \Lambda, y \ne z$, then we cannot have $z \in B_n(y, 4\vp)$ for small enough $\vp$.

So for a point $y \in V$, we know that any two of its different $n$-preimages must belong to distinct balls of type $B_n(\zeta, 2\vp), \zeta \in F_n(\vp)$; and $y$ must have $d^n$ $n$-preimages in  $U$.
If $y_{-n}$ is an $n$-preimage in $U$ of $y$, then there exists $\hat \xi \in \hat \Lambda$ so that $y \in W^u_\vp(\hat \xi)$ and thus $y_{-n} \in B_n(\xi_{-n}, \vp)$ for some $\xi_{-n} \in \Lambda$. But since $F_n(\vp)$ is a maximal $(n, \vp)$-separated set in $\Lambda$, it follows that $\xi_{-n} \in B_n(z, 2\vp)$ for some $z \in F_n(\vp)$. Hence $y_{-n} \in B_n(z, 3\vp)$ and $y \in f^n(B_n(z, 3\vp))$ for some $z \in F_n(\vp)$. Thus we have the following geometric picture of the dynamics on the basin $V$ of the repellor: through every point $y \in V$ there pass $d^n$ tubular neighbourhoods of type $f^nB_n(z_i, 3\vp), z_i \in F_n(\vp), i = 1, \ldots, d^n$. Let us denote such an intersection by $V_n(z_1, \ldots, z_{d^n})$.

Therefore from Lemma \ref{tub} it follows that, if we add the volumes of all sets $f^n(B_n(z, 3\vp)), z \in F_n(\vp)$, we obtain that each piece $V_n(z_1, \ldots, z_{d^n})$ is repeated at least $d^n$ times, hence
$$
d^n m(V) \le \mathop{\sum}\limits_{z \in F_n(\vp)} e^{S_n\Phi^s(z)}
$$
Thus since this happens for any maximal $(n, \vp)$-separated set $F_n(\vp)$,
\begin{equation}\label{eq1}
m(V) \le P_n(\Phi^s- \log d, \vp),
\end{equation}

where $P_n(\psi, \vp)$ denotes in general the quantity $\inf\{\mathop{\sum}\limits_{z \in F} e^{S_n \psi(z)}, F
(n, \vp)-\text{separated in} \ \Lambda\}$, for $\psi$ a continuous real function on $\Lambda$.

Since $V$ is a neighbourhood of $\Lambda$ and thus $m(V) >0$, we obtain that
$$P(\Phi^s - \log d) \ge 0$$

We prove now the opposite inequality.
Indeed let us take some maximal $(n, \vp)$-separated set $F_n(\vp)$ in $\Lambda$ (with respect to $f$). Let a point $y \in V$, where
the neighbourhood $V$ of $\Lambda$ was constructed earlier in the proof.
Then similar to the above proof of the first inequality, each $n$-preimage $y_{-n}^i$ of $y$ must belong to some Bowen ball $B_n(z^i, 3\vp), z^i \in F_n(\vp), i = 1, \ldots, d^n$; hence $y$ belongs to the (open) intersection of $d^n$ tubular unstable neighbourhoods centered at points
$f^n(z_i), z_1, \ldots, z_{d^n} \in F_n(\vp)$, i.e $y \in \mathop{\cap}\limits_{1\le i \le d^n} f^n(B_n(z_i, 3\vp))$.
If $y$ would belong also to some additional tubular unstable neighbourhood $f^n(B_n(\omega, 3\vp))$ for some
$\omega \in F_n(\vp)$, besides the $d^n$ neighbourhoods $f^n(B_n(z_i, 3\vp)), i = 1, \ldots, d^n$, then it
 would follow that $y$ has an additional $n$-preimage $y_{-n}^{d^n+1} \in B_n(\omega, 3\vp)$. Thus since
$B_n(\omega, 3\vp) \subset U$ for small $\vp>0$ and for $\omega \in \Lambda$, we would get a contradiction since $y$ has at most $d^n$ $n$-preimages in $U$; here we used ahain that $\Lambda$ does not intersect the critical set of $f$.
So any $y \in V$ belongs to only $d^n$ tubular unstable neighbourhoods of type $f^n(B_n(z^i, 3\vp)), i = 1, \ldots, d^n$.

Now, as we see from Lemma \ref{tub}, the Lebesgue measure of a tubular unstable neighbourhood
$f^n(B_n(z, 3\vp)), z \in \Lambda$ is comparable to $e^{S_n \Phi^s(z)}$ (where by \textit{comparable}
we mean that the ratio of the two quantities is bounded below and above by positive constants which are independent of
$z, n$). Hence we showed that by taking $\mathop{\sum}\limits_{z \in F_n(\vp)} e^{S_n\Phi^s(z)}$ we cover in fact a combined volume which is less than $C d^n \cdot m(U)$ (for some positive constant $C$ independent of $n$).
From this observation it follows that
$$P(\Phi^s- \log d) \le \mathop{\inf}\limits_n \frac 1n m(U) = 0$$

Combining the two inequalities proved above, we obtain that $P(\Phi^s - \log d) = 0$.
\end{proof}

We are now ready to prove the main result of the paper, namely the existence of a physical measure for the local inverse iterates in the neighbourhood $V$ of the hyperbolic repellor $\Lambda$. We recall that the endomorphism $f$ is \textbf{not} assumed to be  expanding on $\Lambda$, instead it has both stable and unstable directions on $\Lambda$. As seen earlier, we can restrict without loss of  generality to connected repellors. Recall also that we assumed that the critical set of
  $f$ does not intersect $\Lambda$.

\begin{thm}\label{physical}
Let $\Lambda$ be a connected hyperbolic repellor for a smooth
endomorphism $f:M \to M$. There exists a neighbourhood $V$ of
$\Lambda$, $V \subset U$ such that if we denote by $$ \mu^z_n:=
\frac 1n \mathop{\sum}\limits_{y \in f^{-n}z \cap U}
\frac{1}{d(f(y))\cdot \ldots \cdot d(f^n(y))}
\mathop{\sum}\limits_{i=1}^n \delta_{f^i y}, z \in V $$

where $d(y)$ is the number of $f$-preimages belonging to $U$ of a point $y \in V$, then for any continuous function $g \in \mathcal{C}(U, \mathbb R)$ we have
$$
\int_V |\mu^z_n(g) - \mu_s(g)| dm(z) \mathop{\to}\limits_{n \to \infty} 0,
$$
where $\mu_s$ is the equilibrium measure of the stable potential $\Phi^s(x):= \log |\text{det}(Df_s(x))|, x \in \Lambda$.

\end{thm}

\begin{proof}

We assume that $U$ is the neighbourhood of $\Lambda$ from Definition \ref{repellor}, i. e such that $\bar U \subset f(U)$.
As we proved in Proposition \ref{constant}, if $\Lambda$ is a connected hyperbolic repellor, then any point from $\Lambda$ has exactly $d$ $f$-preimages belonging to $\Lambda$ for some positive integer $d$.
Moreover as was shown in the beginning of the proof of Theorem \ref{pressure}, there exists a neighbourhood $V$ of $\Lambda$ such that any point from $V$ has $d^n$ $n$-preimages in $U$, for $n\ge 1$.

If $\Lambda$ is a hyperbolic repellor we have that all local stable manifolds must be contained in $\Lambda$. Indeed, otherwise there may exist small local stable manifolds which are not entirely contained in $\Lambda$. Let $W^s_r(x), x \in \Lambda$ one such stable manifold, with a point $y \in W^s_r(x) \setminus \Lambda$; in this case since $y \in U$ (for small $r$) and since $\bar U \subset f(U)$, it follows that $y$ has a full prehistory $\hat y$ in $U$, and from the fact that $\Lambda$ is a basic set, we obtain that $y \in W^u_r(\hat \xi)$ for some $\hat \xi \in \hat \Lambda$. But then $y = W^s_r(x) \cap W^u_r(\hat \xi)$, hence $y \in \Lambda$ from the local product structure of $\Lambda$ (since $\Lambda$ is a basic set, see for example \cite{KH}); this gives a contradiction to our assumption. Hence there exists a small $r>0$ such that all stable manifolds of size $r$ are contained in $\Lambda$.

We shall denote by $\mathcal{C}(U)$ the space of real continuous functions on $U$.
Let us fix now a Holder continuous function $g \in \mathcal{C}(U)$.
We will apply the $L^1$ Birkhoff Ergodic Theorem (\cite{Ma}) on $\hat \Lambda$ for the homeomorphism $\hat f^{-1}$, in order to obtain an estimate for the measure of the set of prehistories which are badly behaved.
Similarly as in \cite{KH} or \cite{M} we know that the stable distribution is Holder continuous, hence the stable potential on $\hat \Lambda$ is Holder too. This means that there exists a unique equilibrium measure for this potential on $\hat \Lambda$; so from the bijection between $\mathcal{M}(f)$ and $\mathcal{M}(\hat f)$ it follows that there exists a unique equilibrium measure for $\Phi^s$ on $\Lambda$ denoted by $\mu_s$. This measure is ergodic and we can apply the $L^1$ Birkhoff Ergodic Theorem to the function $g\circ \pi$ on $\hat\Lambda$:

\begin{equation}\label{Von}
||\frac 1 n (g(x) + g\circ \pi(\hat f^{-1}(\hat x)) + \ldots g\circ \pi (\hat f^{-n+1}(\hat x)) - \int_\Lambda g\circ \pi d\hat \mu_s||_{L^1(\hat \Lambda, \hat \mu_s)} \mathop{\to}\limits_{n \to \infty} 0
\end{equation}

We make now the general observation that if $f: \Lambda \to \Lambda$ is a continuous map on a compact metric space $\Lambda$, $\mu$ is an $f$-invariant borelian probability measure on $\Lambda$ and $\hat \mu$ is the unique $\hat f$-invariant probability measure on $\hat \Lambda$ with $\pi_*(\hat \mu) = \mu$, then for an arbitrary closed set $\hat F \subset \hat \Lambda$, we have that

\begin{equation}\label{mu-hat}
\hat \mu(\hat F) = \mathop{\lim}\limits_n \mu(\{x_{-n}, \exists \hat x = (x, \ldots, x_{-n}, \ldots) \in \hat F \})
\end{equation}

Let us prove (\ref{mu-hat}): first denote $\hat F_n:= \hat
f^{-n}\hat F, n \ge 1$; next notice that $\hat \mu(\hat F_n) =
\hat \mu(\hat F)$ since $\hat \mu$ is $\hat f$-invariant. Let also
$\hat G_n:= \pi^{-1}(\pi(\hat F_n)), n \ge 1$. We have $\hat F
\subset \hat f^n(\hat G_n), n \ge 0$. Let now a prehistory $\hat z
\in \mathop{\cap}\limits_{n \ge 0} \hat f^n\hat G_n$; then if
$\hat z = (z, z_{-1}, \ldots, z_{-n}, \ldots)$, we obtain that

$z_{-n} \in \pi\hat F_n, \forall n \ge 0$, hence $\hat z \in \hat
F$ since $\hat F$ is assumed closed. Thus we obtain $\hat F =
\mathop{\cap}\limits_{n\ge 0}\hat f^n(\hat G_n)$. Now the above
intersection is decreasing, since $\hat f^{n+1}\hat G_{n+1}
\subset \hat f^n\hat G_n, n \ge 0$. Since the above intersection
is decreasing, we get that $\hat\mu(\hat F) =
\mathop{\lim}\limits_n\hat \mu(\hat f^n\hat G_n) =
\mathop{\lim}\limits_n\hat \mu(\hat G_n) = \mathop{\lim}\limits_n
\hat \mu(\pi^{-1}(\pi(\hat F_n))) = \mathop{\lim}\limits_n
\mu(\pi(\hat F_n)) = \mathop{\lim}\limits_n \mu(\pi\circ \hat
f^{-n} \hat F)$, since $\hat \mu$ is $\hat f$-invariant. Therefore
we obtain (\ref{mu-hat}).

For a positive integer $n$, a continuous real function $g$ defined on the neighbourhood $U$ of $\Lambda$, and a point $y$ so that $y, f(y), \ldots, f^{n-1}(y)$ are all in $U$, let us denote by
$$
\Sigma_n(g, y):= \frac {g(y) + \ldots + g(f^{n-1}y)}{n} - \int g d\mu_s, n \ge 1, y \in \Lambda
$$
Now, from the convergence in $L^1(\hat \Lambda, \hat \mu_s)$ norm it follows the convergence in $\hat \mu_s$-measure. Thus from (\ref{Von}) and (\ref{mu-hat}) we obtain that for any small $\eta > 0$ and $\chi>0$, there exists $N(\eta, \chi)$ so that
 \begin{equation}\label{mu-n}
\mu_s(x_{-n} \in \Lambda, |\Sigma_n(g, x_{-n})| \ge \eta) < \chi, \text{for} \ n > N(\eta, \chi)
 \end{equation}

Let us consider some small $\vp > 0$. Recall that for $n \ge 1$ and $y \in \Lambda$, the Bowen ball
$B_n(y, \vp):= \{z  \in M, d(f^i y, f^i z) < \vp, i = 0, \ldots, n-1\}$.

 We shall prove that if $y \in \Lambda$ and $z \in B_n(y, \vp)$ for $n$
large enough, then the behaviour of $\Sigma_n(g, z)$ is similar to that of $\Sigma_n(g, y)$.
More precisely, assume that $\eta>0$ and that $y \in \Lambda$ satisfies $|\Sigma_n(g, y)| \ge \eta$.
Then we will show that there exists $N(\eta) \ge 1$ so that

\begin{equation}\label{umbra}
|\Sigma_n(g, z)| \ge \frac{\eta}{2}, \forall z \in B_n(y, \vp), n > N(\eta)
\end{equation}

Since $g$ was assumed Holder, let us assume that it has a Holder exponent equal to $\alpha$, i.e
$$|g(x)- g(y)| \le C\cdot d(x, y)^\alpha, \forall x, y \in U,$$ where $d(x, y)$ is the Riemannian distance (from $M$) between $x$ and $y$ and $C>0$ is a constant.
The idea now is that, if $z \in B_n(y, \vp)$, then for some time the iterates of $z$ follow the iterates of $y$ close to stable manifolds, and afterwards they follow the iterates of $y$ closer and closer to unstable manifolds. We have in both cases an exponential growth of distances between iterates, and thus we can use the Holder continuity of $g$ on $U$.

If $z \in B_n(y, \vp), y \in \Lambda$ then we either have $z \in W^s_\vp(y) \subset \Lambda$ or there exists a positive distance
between $z$ and the local stable manifold $W^s_\vp(y)$. In the first case there exists some $\lambda_s \in (0, 1)$ such
that $d(f^i z, f^i y) < \lambda_s^i \vp, i = 0, \ldots, n-1$.
This implies that, in the case when $z \in W^s_\vp(y)$, for some $N_0 \ge 1$ we have:

\begin{equation}\label{Ws}
|g(f^{N_0}y) +\ldots + g(f^{n-1}y)- g(f^{N_0}z)- \ldots- g(f^{n-1}z)| \le \lambda_s^{\alpha N_0} \cdot C_0,
\end{equation}
for some constant $C_0 > 0$ independent of $n$.
 If $z \in B_n(y, \vp)$ but $z$ is not necessarily on $W^s_\vp(y)$, then the iterates of $z$ will approach exponentially
  some local unstable manifolds at the corresponding iterates of $y$ and their "projections" on these unstable manifolds
  increases exponentially, up to a maximum value less than $\vp$ (reached at level $n$).
  More precisely there exists some $N_0, N_1 \ge 1$ and some $\lambda \in (\lambda_s, 1)$ such that $d(f^i z, f^i y)
   \le \lambda^i, i = N_0, \ldots, N_1-1$; notice that $N_0, N_1, \lambda$ are independent of $y, z, n$.
   Now if the iterate $f^{N_1} z$ becomes much closer to $W^u_\vp(f^{N_1} y)$ than to $W^s_\vp(f^{N_1} y)$, it follows
   that all the higher order iterates will approach asymptotically the local unstable manifolds and $d(f^j y, f^j z)$
   increases exponentially. We assume that $N_1$ has been taken such that for some $\lambda_u \in
   (\frac{1}{\inf_\Lambda |Df_u|}, 1)$, we have $d(f^j z, f^j y) \le \lambda_u \cdot d(f^{j+1} z, f^{j+1} y), j = N_1,
   \ldots, n-2$. So the maximum such distance is $d(f^{n-1}y, f^{n-1} z)$ and we know that $d(f^{n-1}y, f^{n-1}z) <
   \vp$ since $z \in B_n(y, \vp)$. Hence $$d(f^j z, f^j y) \le \vp \lambda_u^{n-j-1}, j = N_1, \ldots, n-1$$

Let us take now some $N_2\ge 1$ such that $n - N_2 > N_1$; $N_2$ will be determined later.
Thus from the Holder continuity of $g$ on $U$ we obtain (for some
positive constant $C$) that:

\begin{equation}\label{holder}
\aligned &|g(f^{N_0}z) + \ldots + g(f^{N_1-1}z) + g(f^{N_1}z) +
\ldots + g(f^{n-N_2}z) + \ldots +g(f^{n-1}z)– \\ & -g(f^{N_0}y) -
\ldots – g(f^{N_1-1}y) - g(f^{N_1}y) - \ldots - g(f^{n-N_2}y) -
\ldots  - g(f^{n-1} y)| \le \\ & \le C (\lambda^{\alpha N_0} +
\lambda_u^{\alpha N_2}) + 2 N_2 ||g||
\endaligned
\end{equation}

Thus from (\ref{Ws}) and (\ref{holder}) we obtain that, if $z \in
B_n(y, \vp)$ then:

\begin{equation}\label{Sigma-ev}
|\Sigma_n(g, y) - \Sigma_n(g, z)| \le \frac{1}{n} \left[2N_0 ||g|| +
C(\lambda^{\alpha N_0} + \lambda_u^{\alpha N_2}) + 2N_2 ||g||\right]
\end{equation}

From above, $N_0, N_2$ do not depend on $n, y, z$.
Therefore we can choose some large $N(\eta)$ so that $$\frac 1n
(2N_0 ||g|| + C(\lambda^{\alpha N_0} + \lambda_u^{\alpha N_2}) + 2N_2
||g||)< \eta/2,  \text{for} \ n > N(\eta)$$
This means that the relation from (\ref{umbra}) holds.
Let us denote now by:

\begin{equation}\label{In}
I_n(g, x): = \frac{1}{d^n} \mathop{\sum}\limits_{y \in f^{-n}(x)\cap U} |\Sigma_n(g, y)|,
\end{equation}
for a continuous real function $g:U \to \mathbb R$, and $x \in V$. Recall that $V$ is
the neighbourhood of $\Lambda$, $\Lambda \subset V \subset U$, constructed in the proof of
Theorem \ref{physical} so that every point $x \in V$ has $d^n$ $n$-preimages in $U$ for $n \ge 1$.
For a fixed Holder continuous function $g$ and a small $\eta >0$, we will work with $n > N(\eta)$,
where $N(\eta)$ was found above. From (\ref{Sigma-ev}) and the discussion afterwards, we know that
$|\Sigma_n(g, z) - \Sigma_n(g, y)| \le \eta/2$ if $z \in B_n(y, \vp)$ and $y \in \Lambda$.

Let us consider now an $(n, \vp)$-separated set with maximal cardinality in $\Lambda$, denoted by $F_n(\vp)$.
As in the proof of Theorem \ref{pressure} it follows that any point $y \in V$ belongs to $d^n$ tubular neighbourhoods, i.e $f^n(B_n(y_i, 3\vp)), y_i \in F_n(\vp)$ for $1 \le i \le d^n$. Let us denote as before $V_n(y_1, \ldots, y_{d^n}):= \mathop{\cap}\limits_{1\le i \le d^n} f^n B_n(y_i, 3\vp)$.

Thus in $\int_V I_n(g, x) dm(x)$, we can decompose $V$ into the smaller pieces $V_n(y_1, \ldots, y_{d^n})$, for different choices of $y_1, \ldots, y_{d^n} \in F_n(\vp)$.

We can use now relation (\ref{Sigma-ev}) in order to replace in $\int_V I_n(g, x) dm(x)$, the term $|\Sigma_n(g, y)|$ with $|\Sigma_n(g, \zeta)|$, where $x \in V$ is arbitrary, $y \in f^{-n}x \cap U$ and $y \in B_n(\zeta, 3\vp)$ for some $\zeta \in F_n(\vp)$.
Indeed let us fix some arbitrary small $\eta>0$. Then we prove similarly as in (\ref{Sigma-ev})

that if $n > N(\eta)$, then $|\Sigma_n(g, y)| \le |\Sigma_n(g, \zeta)| + \eta/2$, if $y \in B_n(\zeta, 3\vp)$
and $\zeta \in F_n(\vp)$ ($N(\eta)$ can be assumed to be the same as in (\ref{Sigma-ev}) without loss of generality).

So up to a small error of $\eta/2$ we can replace each of the
terms $|\Sigma_n(g, y)|$ with the corresponding $|\Sigma_n(g,
\zeta)|$. This implies that in the integral $\int_V I_n(g,
x)dm(x)$, on each piece of type $V_n(y_1, \ldots, y_{d^n})$ in
$f^n(B_n(y_j, 3\vp))$ for $y_j \in F_n(\vp)$, we integrate in fact
$|\Sigma_n(g, y_j)|$, modulo an error of $\eta/2$. Then we will
obtain that $$ \aligned &\int_V I_n(g, x)dm(x) \le \frac {1}{d^n}
\mathop{\sum}\limits_{z_1, \ldots, z_{d^n} \in
F_n(\vp)}\int_{V_n(z_1, \ldots, z_{d^n})}
\mathop{\sum}\limits_{i=1}^n |\Sigma_n(g, z_i)| dm
+\frac{\eta}{2}\cdot m(V) \\ &\le \frac{1}{d^n} \mathop{\sum}
\limits_{z \in F_n(\vp)} |\Sigma_n(g, z)| \cdot
\mathop{\sum}\limits_{z \in \{z_1, \ldots, z_{d^n}\}} m(V_n(z_1,
\ldots, z_{d^n})) + m(V) \eta/2 \\ &\le \frac{1}{d^n}
\mathop{\sum}\limits_{z \in F_n(\vp)} |\Sigma_n(g, z)| \cdot m(f^n
B_n(z, 3\vp)) + m(V) \eta/2
\endaligned
$$
So what we did is, we replaced $|\Sigma_n(g, y)|$ with $|\Sigma_n(g, z)|$ for all $y \in f^{-n}x \cap U$, where
$y \in B_n(z, 3\vp), z\in F_n(\vp)$, then we integrated the respective sums of $|\Sigma_n(g, z)|, z \in F_n(\vp)$ on small
pieces of tubular overlap $V_n(z_1, \ldots, z_{d^n})$; lastly, we kept $|\Sigma_n(g, z)|$ fixed for an arbitrary $z \in F_n(\vp)$ and added the measures of all intersections of $f^n B_n(z, 3\vp)$ with other tubular sets of type $f^n B_n(w, 3\vp), w \in F_n(\vp)$.
Thus by adding the measures of these overlaps, we recover $m(f^n B_n(z, 3\vp))$.
In conclusion we obtain:

\begin{equation}\label{int}
\int_V I_n(g, x) dm(x) \le C \cdot \mathop{\sum}\limits_{y \in F_n(\vp)} |\Sigma_n(g, y)| \cdot \frac{m(f^n(B_n(y, 3\vp))}{d^n} + \frac{\eta}{2} \cdot m(V)
\end{equation}

We recall now from Lemma \ref{tub}, that $m(f^n(B_n(y, 3\vp)))$ is comparable to
$e^{S_n\Phi^s(y)}$, independently of $n, y \in \Lambda$. And from Theorem \ref{pressure} we know that $P(\Phi^s) = \log d$.
Hence from Proposition \ref{eq-endo} we have that, if $\mu_s$ denotes the unique equilibrium measure of $\Phi^s$,
then $\mu_s(B_n(y, \vp/2))$ is comparable to $\frac{e^{S_n \Phi^s(y)}}{d^n}$, independently of $n, y$.
Therefore combining with (\ref{int}) we obtain that there exists a constant $C_1 >0$ s.t:

\begin{equation}\label{suma}
\int_V I_n(g, x) dm(x) \le C_1 \left(\mathop{\sum}\limits_{y \in F_n(\vp)} |\Sigma_n(g, y)| \mu_s(B_n(y, \vp/2)) + \eta \right),
\end{equation}
for $n > N(\eta)$. We will split now the points of $F_n(\vp)$ in
two disjoint subsets denoted by $G_1(n, \vp)$ and $G_2(n, \vp)$,
defined as follows: $$ G_1(n, \vp):= \{ y \in F_n(\vp),
|\Sigma_n(g, y)| < \eta\} \ \text{and} \ G_2(n, \vp):= \{z \in
F_n(\vp), |\Sigma_n(g, z)| \ge \eta\} $$

Recall that the Bowen balls $B_n(y, \vp/2), y \in F_n(\vp)$ are
mutually disjointed since $F_n(\vp)$ is $(n, \vp)$-separated. Also if $y \in G_2(n, \vp)$ and $z \in B_n(y, \vp/2)$,
we have $|\Sigma_n(g, z)| \ge \eta/2$ (from (\ref{umbra}));
hence $B_n(y, \vp/2)\cap \Lambda \subset \{z \in \Lambda, |\Sigma_n(g, z)| \ge \eta/2\}$.
Consequently for a constant $C_\vp>0$,
$$
\aligned
&\mathop{\sum}\limits_{y \in F_n(\vp)} |\Sigma_n(g, y)| \mu_s(B_n(y, \vp/2)) =
\mathop{\sum}\limits_{y \in G_1(n, \vp)} |\Sigma_n(g, y)| \mu_s(B_n(y, \vp/2)) +
\mathop{\sum}\limits_{y \in G_2(n, \vp)} |\Sigma_n(g, y)| \mu_s(B_n(y, \vp/2)) \le \\
& \le \eta \mathop{\sum}\limits_{y \in G_1(n, \vp)} \mu_s(B_n(y, \vp/2)) + 2||g||
\mu_s(z \in \Lambda, |\Sigma_n(g, z)| \ge \frac{\eta}{2})\cdot C_\vp
\endaligned
$$

But since the balls $B_n(y, \vp/2), y \in F_n(\vp)$ are mutually disjoint, we have $\mathop{\sum}\limits_{y \in G_1(n, \vp)} \mu_s(B_n(y, \vp/2)) \le 1$. Also $\mu_s(z \in \Lambda,
|\Sigma_n(g, z)| \ge \eta/2) < \chi$ for $n > N(\eta/2, \chi)$, as follows from (\ref{mu-n}).
Thus by using (\ref{suma}) we obtain for $n > \sup\{N(\eta), N(\eta, \chi)\}$
$$
\int_V I_n(g, x) dm(x) \le C_1 (\eta + \eta + C_\vp \cdot 2||g|| \chi) = 2C_1 (\eta + \chi \cdot C_\vp ||g||)
$$
Since $\eta, \chi>0$ were taken arbitrarily, and by recalling the formula for $I_n(g, x)$ from (\ref{In}) and the definition of $\mu^z_n$, we obtain that:
$$
\int_V |\mu^z_n(g) - \mu_s(g)| dm(z) \mathop{\to}\limits_{n \to \infty} 0
$$
Since Holder continuous functions $g$ are dense in the uniform norm on $\mathcal{C}(U)$, we obtain the conclusion of the Theorem for all $g \in \mathcal{C}(U)$.

\end{proof}

\begin{cor}\label{subseq}

In the same setting as in Theorem \ref{physical}, it follows that there exists a borelian set $A \subset V$
with $m(V \setminus A) = 0$ and a subsequence $(n_k)_k$ such that $\mu_{n_k}^z \mathop{\to} \limits_{k \to \infty}
\mu_s$ (as measures on $U$), for any point $z \in A$.
\end{cor}

\begin{proof}

Let us fix $g \in \mathcal{C}(U)$. From the
convergence in Lebesgue measure of the sequence of functions $z
\to \mu^z_n(g), n \ge 1, z \in V$ obtained from Theorem \ref{physical}, it
follows that there exists a borelian set $A(g)$ with
$m(V \setminus A(g)) = 0$ and a subsequence $(n_p)_p$ so that
$\mu^z_{n_p}(g) \mathop{\to}\limits_p  \mu_s(g), z \in A(g)$.
Let us consider now a sequence of functions $(g_m)_{m \ge 1}$ dense in
$\mathcal{C}(U)$. By applying a diagonal sequence
procedure we obtain a subsequence $(n_k)_k$ so that
$\mu^z_{n_k}(g_m) \mathop{\to}\limits_k  \mu_s(g_m), \forall z \in
\mathop{\cap}\limits_m A(g_m), \forall m \ge 1$. We have also
$m (V \setminus \mathop{\cap}\limits_m A(g_m)) = 0$, since
$m (V \setminus A(g_m)) = 0, m \ge 1$.
However any real continuous function $g \in \mathcal{C}(U)$ can be approximated in the
uniform norm by functions $g_m$, hence it follows that
$\mu^z_{n_k}(g) \mathop{\to}\limits_k \mu_s(g), \forall z \in
A:= \mathop{\cap}\limits_m A(g_m)$. But we showed above that $m(V \setminus A) = 0$.

So we obtain that $\mu^z_{n_k} \mathop{\to}\limits_k  \mu_s$ for all points $z \in A$, where
$A$ has full Lebesgue measure in $V$.

\end{proof}

\section{Applications. Examples.}\label{sec3}

In this section we will pursue further ergodic properties of the inverse physical measure constructed
in Theorem \ref{physical}
and give also examples. Let us first remind the notion of the Jacobian of an endomorphism, relative to an invariant probability measure, from
Parry's book (\cite{P}).
Let $f: (X, \mathcal{B}, \mu) \to (X, \mathcal{B}, \mu)$ a measure preserving endomorphism on a Lebesgue probability space.
Assume that the fibers of $f$ are countable, i.e $f^{-1}x$ is countable for $\mu$-almost all $x \in X$. It can be proved (\cite{P}) that in this case $f$ is positively non-singular, i.e $\mu(A) = 0$ implies $\mu(f(A)) = 0$ for an arbitrary measurable set $A \subset X$. Also there exists a measurable partition $\alpha = (A_0, A_1, \ldots)$ of $X$ such that $f|_{A_i}$ is injective. Then using the
absolute continuity of $\mu\circ f$ with respect to $\mu$, we define the \textbf{Jacobian} $J_{f, \mu}$ on each set $A_i$, to be equal to the Radon-Nikodym derivative $\frac{d\mu\circ f}{d \mu}$. So:
$$
J_{f, \mu} (x) = \frac{d\mu\circ f}{d\mu}(x), x \in A_i, i \ge 0
$$
This is a well defined measurable function, which is larger or equal than 1 everywhere
(due to the $f$-invariance of $\mu$). Also it is easy to see that $J_{f, \mu}(\cdot)$ is
independent of the partition $\alpha$ and that it satisfies a Chain Rule, namely $J_{f\circ g, \mu} = J_{f, \mu}
\cdot J_{g, \mu}$ if $f, g:X \to X$ and both preserve $\mu$.
From Lemma 10.5 of \cite{P} we also know that $$\log J_{f, \mu}(x) = I(\epsilon/f^{-1}\epsilon)(x),$$
for $\mu$-almost every $x \in X$, where $\epsilon$ is the partition of $X$ into single points, and
$I(\epsilon/f^{-1}\epsilon)(\cdot)$ is the conditional information function of $\epsilon$ given the partition $f^{-1}\epsilon$. Also from the definition of the Jacobian we see (\cite{PDL}) that:

\begin{equation}\label{jac}
\mu(fA) = \int_A J_{f, \mu}(x) d\mu(x),
\end{equation}

for all \textit{special sets} $A$, i.e measurable sets such that $f|_A: A \to f(A)$ is injective.
Recall that by Definition \ref{repellor}, $f$ does not have any critical points in $\Lambda$.
Before proving the main result of this Section, we remind the notion of measurable partitions subordinated to local stable manifolds; for background on measurable partitions, Lebesgue spaces and conditional measures, one can use \cite{Ro}.

Let $f:M \to M$ be a smooth endomorphism defined on a Riemannian manifold $M$ which is endowed with its Borelian $\sigma$-algebra $\mathcal{B}$. Let also a probability borelian measure $\mu$ on $M$ which is $f$-invariant.
 If $\xi$ is a measurable partition of $M$, then we denote by $\xi(x)$ the unique subset of $\xi$ containing $x \in X$; also by $(M/\xi, \mu_\xi)$ we denote the factor space relative to $\xi$. To any measurable partition $\xi$ on $(M, \mathcal{B}, \mu)$ one can attach an essentially unique collection of \textbf{conditional measures} $\{\mu_C\}_{C \in \xi}$ satisfying two conditions (see \cite{Ro}):

 \ \ i) $(C, \mu_C)$ is a Lebesgue space

 \ \ ii) for any measurable set $B \subset M$, the set $B\cap C$ is measurable in $C$ for almost all points $C \in M/\xi$
 of the factor space, and the function $C \to \mu_C(B \cap C)$ is measurable on $M/\xi$ and $\mu(B) =
 \int_{M/\xi} \mu_C(B \cap C) d\mu_\xi$.

Similar to the case of partitions subordinated to unstable
manifolds (\cite{Y}) we can say (as in \cite{PDL}), that a
measurable partition $\xi$ of $(M, \mathcal{B}, \mu)$ is
\textbf{subordinate to local stable manifolds} if for $\mu$-almost
all $x \in M$ one has $\xi(x) \subset W^s_r(x)$ and if $\xi(x)$
contains an open neighbourhood of $x$ inside $W^s_r(x)$ (where
$r>0$ is sufficiently small). We can define now the absolute
continuity of conditional measures on stable manifolds as in
\cite{PDL}:

\begin{defn}\label{accm}

In the above setting, we say that $\mu$ has \textbf{absolutely
continuous conditional measures on local stable manifolds} if for
every measurable partition $\xi$ subordinated to local stable
manifolds, we have for $\mu$ almost all $x \in M$ that $\mu^\xi_x
\ll m^s_x$, where $\mu^\xi_x$ is the conditional measure of $\mu$
on $\xi(x)$ and $m^s_x$ denotes the induced Lebesgue measure on
$W^s_r(x)$.

\end{defn}

By the result of Liu (\cite{PDL}), we know that there exists at least one measurable partition
subordinated to local stable manifolds.

Now, by Oseledec Theorem (\cite{Ma}) we have that for any $f$-invariant Borel probability measure $\mu$ on $M$, and
for $\mu$-almost every point $x \in M$ there exists a finite collection of numbers, called \textit{Lyapunov exponents} of
$f$ at $x$ with respect to $\mu$,
$
-\infty \le \lambda_1(x) < \lambda_2(x) < \ldots < \lambda_{q(x)}(x) < \infty,
$ and a unique collection of tangent subspaces of $T_xM$, $V_1(x) \subset \ldots \subset V_{q(x)}(x) = T_x M$ so that
$$
\mathop{\lim}\limits_n \frac 1n \log |Df^n_x (v)| = \lambda_i(x), \forall v \in V_i(x)
\setminus V_{i-1}(x), 1 \le i \le q(x), |v| = 1
$$
We also denote by $m_i(x):= \text{dim}V_i(x) - \text{dim} V_{i-1}(x)$ the \textit{multiplicity} of $\lambda_i(x)$. As we saw before, if $\Lambda$ is a connected repellor for $f$ then $f|_\Lambda$ is constant-to-1.
We are now ready to prove the following:

\begin{thm}\label{cond}

Let $\Lambda$ be a connected hyperbolic repellor for a smooth endomorphism $f:M \to M$ on a Riemannian manifold $M$; assume that $f$ is $d$-to-1 on $\Lambda$.
Then there exists a unique $f$-invariant probability measure $\mu^-$ on $\Lambda$ satisfying an inverse Pesin entropy formula:
$$
h_{\mu^-}(f) = \log d - \int_\Lambda \mathop{\sum}\limits_{i, \lambda_i(x) < 0} \lambda_i(x) m_i(x) d\mu^-(x)
$$
In addition the measure $\mu^-$ has absolutely continuous conditional measures on local stable manifolds.
\end{thm}

\begin{proof}

Notice that from the above properties of Lyapunov exponents, the derivative
$Df^n_{s,x}$ for large $n$, takes into consideration all the vectors $v \in V_i(x)$ for those
$i$ for which $\lambda_i(x) <0$, i.e for which we have contraction in the long run. Thus if $\mu$ is an $f$-invariant probability measure supported on $\Lambda$, we have by the Chain Rule and Birkhoff
Theorem that

\begin{equation}\label{Lyapunov}
\aligned
\int_\Lambda \Phi^s d\mu & =  \int_\Lambda \mathop{\lim}\limits_n \frac 1n \mathop{\sum}\limits_{i = 0}^{n-1} \Phi^s(f^i x)
d\mu(x) = \\
& = \int_\Lambda \mathop{\lim}\limits_n \frac 1n \log |\text{det}(Df^n_{s, x})| d\mu(x) = \int_\Lambda
\mathop{\sum}\limits_{i, \lambda_i(x)<0} \lambda_i(x) m_i(x)d\mu(x)
\endaligned
\end{equation}

It follows that the inverse Pesin entropy formula from the statement of the Theorem is satisfied for $\mu= \mu_s$ since $\mu_s$ is the equilibrium measure of $\Phi^s$ and we showed in Theorem \ref{pressure} that $P(\Phi^s - \log d) = 0$. If the inverse Pesin entropy formula would be satisfied for another invariant measure $\mu$, then  we would have $h_{\mu}(f) = \log d - \int_\Lambda \mathop{\sum}\limits_{i, \lambda_i(x) < 0}
\lambda_i(x) m_i(x) d\mu(x)$, hence:
$$P(\Phi^s - \log d) \ge h_\mu - \log d + \int_\Lambda \Phi^s d\mu = 0$$

However again from Theorem \ref{pressure} we know that
$P(\Phi^s - \log d) = 0,$ thus
$\mu$ is an equilibrium measure for $\Phi^s$. But $\Phi^s$ is Holder continuous  and thus it has a unique equilibrium measure.
Therefore if $\mu^- := \mu_s$, we have $$\mu = \mu_s = \mu^-$$

We want now to show the absolute continuity of conditional measures of $\mu^-$ on local stable manifolds. For this we will use Corollary  \ref{subseq} and  results from \cite{PDL}.
Indeed we know that $\Lambda$ is a connected hyperbolic repelor and thus $f$ is $d$-to-1 for some integer
$d \ge 1$ in a neighbourhood $V$ of $\Lambda$. We constructed the measures $\mu^z_n, z \in V, n \ge 1, \mu^z_n :=
\frac {1}{d^n} \mathop{\sum}\limits_{y \in f^{-n}z} \frac{1}{n} \mathop{\sum}\limits_{i=1}^n \delta_{f^iy}$;
and we showed in Corollary \ref{subseq} that there exists a subset $A \subset V$, having full Lebesgue measure and a
subsequence $(\mu^z_{n_k})_k$ converging weakly towards $\mu^- := \mu_s$ for every $z \in A$. Now in (\ref{jac}) we can take only
special sets whose boundaries have $\mu^-$-measure equal to zero. For such a set $B$ we have that $\mu_{n_k}(B)
\mathop{\to}\limits_{k \to \infty} \mu^-(B)$. But then from the definition of $\mu^z_n$ it follows that
$\mu^-(f(B)) = d \mu^-(B)$ for any such special set with boundary of measure zero.
As these sets form a sufficient collection (\cite{KH}), we obtain that the Jacobian $J_{f, \mu^-}$ is constant
$\mu^-$-almost everywhere and equal to $d$.
Hence from Lemma 10.5 of \cite{P}, if $\epsilon$ denotes the partition of $M$ into single points, we deduce that
the conditional information function
$I(\epsilon/f^{-1}\epsilon)(x) = \log J_{f, \mu^-}(x) = \log d$ for $\mu^-$-almost all $x \in \Lambda$; thus
$$
H_{\mu^-}(\epsilon/f^{-1}\epsilon) = \int I(\epsilon/f^{-1}\epsilon)(x) d\mu^-(x) = \log d
$$

Then since $h_{\mu^-} = \log d - \int_\Lambda \mathop{\sum}\limits_{\lambda_i(x) <0} \lambda_i(x) d \mu^-(x)$, it follows that
$$
h_{\mu^-} = H_{\mu^-}(\epsilon/f^{-1}\epsilon) - \int_\Lambda \mathop{\sum}\limits_{i, \lambda_i(x) < 0} \lambda_i(x) m_i(x) d\mu^-(x)
$$
Hence from \cite{PDL} we obtain that $\mu^-$ has absolutely continuous conditional measures on local stable manifolds.
\end{proof}

The question whether a measure-preserving dynamical system is 2-sided or 1-sided Bernoulli is an important one and has been solved in a number of cases (see for example \cite{Ma}, \cite{BH}, \cite{Ru-exp}). In our case we show that the inverse SRB measure $\mu^-$ on the repellor $\Lambda$, does not form a 1-sided Bernoulli system, by contrast with the usual SRB measures on attractors of diffeomorphisms (which is 2-sided Bernoulli).

\begin{thm}\label{Bernoulli}

Let $f$ as above and $\Lambda$ a connected repellor as in Theorem
\ref{cond} so that $f$ is not invertible on $\Lambda$. Then
$(\Lambda, f|_\Lambda, \mathcal{B}(\Lambda), \mu^-)$ cannot be
one-sided Bernoulli.

\end{thm}

\begin{proof}

Let $(\Sigma_m^+, \sigma_m, \mathcal{B}_m, \mu_p)$ a one-sided Bernoulli shift on $m$ symbols (\cite{Ma}),
where $\mathcal{B}_m$ denotes the $\sigma$-algebra of sets generated by cylinders in $\Sigma_m^+$,
$\sigma_m$ is the shift map, and $\mu_p$ is the $\sigma_m$-invariant measure associated to a probability
vector $p = (p_1, \ldots, p_m)$.

We know from Proposition \ref{constant} that if $\Lambda$ is
connected, then the number of $f$-preimages belonging to $\Lambda$
is constant, and denote it by $d$; we assumed that $d >1$. If $(\Lambda, f|_\Lambda,
\mathcal{B}(\Lambda), \mu^-)$ would be isomorphic to a one-sided
Bernoulli system $(\Sigma_m^+, \sigma_m, \mathcal{B}_m, \mu_p)$,
then $d = m $ since the number of preimages is constant
everywhere, for both systems. But then from the Variational Principle for entropy, we would obtain:
\begin{equation} \label{Ber}
h_{\mu^-} = h_{\mu_p} \le h_{top}(\sigma_m) = \log m = \log d
\end{equation}

On the other hand since $\mu^-$ satisfies the Pesin formula on
$\Lambda$, we get that $h_{\mu^-} = \log d - \int\Phi^s d\mu^-$.
But $\Phi^s <0$ and $\mathcal{C}_f \cap \Lambda = \emptyset$,
hence $h_{\mu^-} > \log d$. This gives a contradiction to
(\ref{Ber}).

\end{proof}

We prove now that, in spite of not being 1-sided Bernoulli, the
inverse SRB measure $\mu^-$ has strong mixing properties on the
repellor $\Lambda$.

Given a transformation $f:M \to M$ we say that an $f$-invariant
probability $\mu$ has \textbf{Exponential Decay of Correlations}
on Holder potentials (\cite{Bo}) if there exists some $\lambda \in
(0, 1)$ such that for every $n \ge 1$: $$ |\int \phi \cdot \psi
\circ f^n d\mu - \int \phi d\mu \cdot \int \psi d\mu| \le C(\phi,
\psi) \lambda^n, $$ for any Holder maps $\phi, \psi \in
\mathcal{C}(M, \mathbb R)$, where $C(\phi, \psi)$ depends only on
the potentials $\phi, \psi$.

\begin{thm}\label{exact}

Let a repellor $\Lambda$ for a smooth endomorphism as in Theorem \ref{physical} and let $\mu^-$ be the unique inverse SRB measure associated. Then
$\mu^-$ has Exponential Decay of Correlations on Holder potentials.

\end{thm}

\begin{proof}

Since we have a uniformly hyperbolic structure for the endomorphism $f$ on $\Lambda$, we can associate to it a Smale
space structure on the natural extension $\hat \Lambda$ (\cite{Ru-carte}). Therefore on $\hat \Lambda$ there exist
 Markov partitions of arbitrarily small diameter (\cite{Ru-carte}). Now these Markov partitions imply the existence
 of a semi-conjugacy $h$ with a 2-sided mixing Markov chain $\Sigma_A$. We have therefore the Lipschitz continuous maps
$$
h: \Sigma_A \to \hat \Lambda, \text{and} \  \pi: \hat \Lambda \to \Lambda
$$
such that $\pi \circ \hat f = f \circ \pi, h \circ \sigma_A = \hat f \circ h$, where $\sigma_A$ is the shift homeomorphism.

Now, since the stable potential $\Phi^s$ on $\Lambda$ is Holder continuous,
it follows that $\Psi^s:= \Phi^s \circ \pi \circ h: \Sigma_A \to \mathbb R$ is Holder continuous
 and to the unique equilibrium measure $\mu_s$ of $\Phi^s$ it corresponds the
unique equilibrium measure $\nu$ of $\Psi^s$ on $\Sigma_A$, s.t $\mu_s = (\pi \circ h)_* \nu$. We have that $P_f(\Phi^s) = P_{\sigma_A}(\Psi^s)$ and $h_{\mu_s}(f) = h_{\nu}(\sigma_A)$.
Also notice that
$$\int_\Lambda \phi d\mu_s = \int _{\Sigma_A} \phi \circ \pi \circ h \ d\nu, \phi \in \mathcal{C}(\Lambda)
$$
Now we do have Exponential Decay of Correlations for Holder potentials for $(\Sigma_A, \nu)$ (for example \cite{Bo}); so the same holds for $f|_\Lambda$  and the equilibrium measure $\mu_s$. Recalling that we denoted $\mu^- := \mu_s$, we obtain the conclusion.

\end{proof}

\textbf{Examples:}

\ \textbf{1. Toral endomorphisms.} Let us take an integer valued $m \times m$ matrix $A$ with $\text{det}(A) \ne 1$.
This matrix induces a toral endomorphism $f_A: \mathbb T^m \to \mathbb T^m$.
This toral endomorphism transforms the unit square into a parallelogram in $\mathbb R^m$ of area (Lebesgue measure)
equal to $|\text{det}(A)|$, and whose corners are points having only integer coordinates.
Thus when we project to $\mathbb T^m$, we obtain that $f_A$ is $|\text{det}(A)|$-to-1.
If all eigenvalues of $A$ have absolute values different from 1, then $f_A$ is hyperbolic on the whole torus
$\mathbb T^m$.

Theorem \ref{physical} can be applied in this case, since $\mathbb T^m$ is a connected hyperbolic repellor for $f_A$, and we obtain a physical measure for the multivalued inverse iterates of $f_A$. In this case the inverse SRB measure $\mu^-$ is in fact the Haar measure on $\mathbb T^m$ since the stable potential is constant. Also from Theorem \ref{cond}, we obtain that a Pesin type formula holds for the negative Lyapunov exponents.

\textbf{2. Anosov endomorphisms.} Theorem \ref{physical} and \ref{cond} can be applied also in the case of Anosov endomorphisms on a Riemannian manifold $M$, since $M$ can be viewed as a hyperbolic repellor.
In general the stable potential is not constant and $\mu^-$ is not necessarily absolutely continuous with respect to the Lebesgue measure on $M$.
We obtain again that the asymptotic distribution of preimages for Lebesgue almost every point in $M$ is equal to the equilibrium measure $\mu^- = \mu_s$, and that the inverse SRB measure $\mu^-$ has absolutely continuous conditional measures on local stable manifolds.

\textbf{3. Non-Anosov hyperbolic non-expanding repellors for products.}
Let us take for instance $f:\mathbb P\mathbb C^1 \to \mathbb P\mathbb C^1, f([z_0: z_1]) = [z_0^2: z_1^2]$, and $g:\mathbb T^2 \to \mathbb T^2, g$ being induced by the matrix $A = \left(\begin{array}{ll}
                    2 & 2 \\

                    2 & 3
                  \end{array} \right)$.
We see easily that $A$ has one eigenvalue in $(0, 1)$ and another larger than 1, so $g$ is hyperbolic. We take the product
$$
F: \mathbb P\mathbb C^1 \times \mathbb T^2 \to \mathbb P\mathbb C^1 \times \mathbb T^2, F([z_0: z_1], (x, y)) = (f([z_0: z_1]), g(x, y)) \ \text{and} \ \Lambda:= S^1 \times \mathbb T^2
$$
Then $\Lambda$ is a connected hyperbolic non-Anosov repellor for the smooth endomorphism $F$ and we can apply Theorems \ref{physical} and \ref{cond}.

\textbf{4. Perturbations.}
According to Proposition \ref{stability}, if $f$ is hyperbolic on a connected repellor $\Lambda$ and if an
endomorphism $g$ is a $\mathcal{C}^1$ perturbation of $f$,
then $g$ has a connected hyperbolic repellor denoted $\Lambda_g$ which is close to $\Lambda$.
We can form then a large class of examples by perturbing known examples, like the ones above.
Then since $g$ is hyperbolic on $\Lambda_g$ we can again apply Theorems \ref{physical} and \ref{cond}, this time for inverse SRB measures which might be more complicated than in the original (unperturbed) example.
For instance, let us take $F: \mathbb P\mathbb C^1 \times \mathbb T^2 \to \mathbb P\mathbb C^1 \times \mathbb T^2$ given by
$$
F([z_0:z_1], (x, y)) = ([z_0^2:z_1^2],  f_A(x, y)),
$$
where $f_A$ is the toral endomorphism induced by the matrix $A=\left(\begin{array}{ll}
                    2 & 1 \\

                    2 & 2

                  \end{array} \right)$.
As can be seen, $F$ has a connected hyperbolic repellor $\Lambda:= S^1 \times \mathbb T^2$.
Consider the following perturbation of $F$,  $F_\vp: \mathbb P\mathbb C^1 \times \mathbb T^2 \to \mathbb P\mathbb C^1 \times \mathbb T^2$ given by:
$$
F_\vp([z_0:z_1], (x, y)) = \left([z_0^2 + \vp z_1^2 \cdot e^{2\pi i (2x+y)}:z_1^2], (2x+y+\vp sin(2\pi(x+y)),
2x+2y+\vp cos^2(4\pi x))\right)
$$
It can be seen that $F_\vp$ is well defined as a smooth endomorphism on
$\mathbb P\mathbb C^1 \times \mathbb T^2$ and that it is a $\mathcal{C}^1$ perturbation of $F$.
It follows from Proposition \ref{stability} that $F_\vp$ has a connected hyperbolic repellor $\Lambda_\vp$ (on which
$F_\vp$ has both stable as well as unstable directions), and that $\Lambda_\vp$ is close to $\Lambda$.
However $\Lambda_\vp$ is different from $\Lambda$, and it has a complicated structure with self-intersections; its
projection on the second coordinate is $\mathbb T^2$.
For this repellor $\Lambda_\vp$ we can apply Theorem \ref{physical} to get a physical measure $\mu^-_\vp$ for the
local inverse iterates of $F_\vp$. This physical measure $\mu^-_\vp$ is the equilibrium measure of the non-constant
stable potential
$$\Phi^s_\vp([z_0:z_1], (x, y)) := \log |\text{det}(DF_\vp)_s([z_0:z_1], (x, y))|, \text{for} \ ([z_0:z_1], (x, y)) \in
\Lambda_\vp$$
We know from Theorem \ref{cond} that the conditional measures of the inverse SRB measure $\mu^-_\vp$ on the local stable
manifolds (which are contained in the repellor $\Lambda_\vp$), are absolutely continuous with respect to
the induced Lebesgue measures.

Also a Pesin type formula is true for the measure-theoretic
entropy $h_{\mu^-_\vp}$ of $\mu^-_\vp$, and the negative Lyapunov
exponents (which are non-constant if $\vp \ne 0$).$\hfill \square$

\

Similarly one can perturb other connected hyperbolic repellors to obtain new
 dynamical systems for which Theorems \ref{physical} and \ref{cond}, as well as Corollary \ref{subseq} can be applied.

Another observation is that one can form repellors quite easily.
We need only the existence of families of stable/unstable cones in
some open set $U$ and the topological condition $\bar U \subset
f(U)$. Then one can form the basic set $\Lambda:=
\mathop{\cap}\limits_{n \in \mathbb Z} f^n(U)$, on which we have a
hyperbolic structure. The inverse SRB measure $\mu^-$ supported on
$\Lambda$  can be approximated Lebesgue almost everywhere on $U$,
like in Theorem \ref{physical}, and will have good ergodic
properties as found in Theorem \ref{exact}. However it may be
difficult to describe this measure explicitly, especially in the
non-Anosov case, since $(\Lambda, \mu^-)$ is not 1-sided
Bernoulli.

Institute of Mathematics of the Romanian Academy, P. O. Box 1-764, RO 014700, Bucharest, Romania.

\textbf{Email:} Eugen.Mihailescu\@@imar.ro

\textbf{Webpage:}  www.imar.ro/~mihailes

\end{document}